\documentclass[leqno, 12pt]{amsart}

\usepackage{numprint}
\usepackage[english]{babel}

\title{Enumeration of double cosets and self-inverse double cosets}
\usepackage[utf8]{inputenc}
\usepackage{geometry}

\usepackage{lineno}

\usepackage{amsmath}
\usepackage{amssymb}
\usepackage{amsthm}

\usepackage{graphicx}
\usepackage{subfig}
\usepackage{caption}
\usepackage{hyperref}
\usepackage{float}
\usepackage[table]{xcolor}
\usepackage{booktabs}

\usepackage{thm-restate}

\numberwithin{equation}{section}
\newtheorem{thm}{Theorem}[section]

\newtheorem{prop}[thm]{Proposition}
\newtheorem{defi}[thm]{Definition}
\newtheorem{cnj}[thm]{Conjecture}

\newcommand{\dps}{\displaystyle}
\newcommand{\Sq}{\operatorname{Sq}}
\newcommand{\GL}{\operatorname{GL}}
\newcommand{\Ss}{\mathfrak S}
\newcommand{\F}{\mathbb F}
\newcommand{\cube}{B}
\newcommand{\octa}{\tilde B}
\newcommand{\diag}{\Delta}
\newcommand{\ctype}{\mathbf t}
\newcommand{\selfinv}{\Theta}
\newcommand{\ncl}{N}
\newcommand{\Poly}{\operatorname{Poly}}
\newcommand{\Cyc}{\operatorname{Cyc}}

\title{On the enumeration of double cosets and self-inverse double cosets}
\author{Ludovic Schwob}

\begin{document}
	\npthousandsep{\hspace{0.5mm}}

\begin{abstract}
	Double cosets appear in many contexts in combinatorics, for example in the enumeration of certain objects up to symmetries. Double cosets in a quotient of the form $H\backslash G / H$  have an inverse, and can be their own inverse.  In this paper we present various formulas enumerating double cosets, and in particular self-inverse double cosets. We study double cosets in classical groups, especially the symmetric groups and the general linear groups, explaining how to obtain the informations on their conjugacy classes required to apply our formulas. We also consider double cosets of parabolic subgroups of Coxeter groups of type B.
\end{abstract}

\maketitle

\tableofcontents
\section{Introduction}
Cosets are fundamental objects in group theory, and their enumeration is easy since all cosets of a subgroup have the same size. A group $G$ with subgroups $H$ and $K$ can also be partitioned into double cosets: $H\backslash G/K=\{HgK:g\in G\}$. Double cosets are much less studied than ordinary cosets, and their enumeration is more complicated as all double cosets do not have the same size.
 Nevertheless, recent papers have shown that double cosets are very interesting objects in algebraic combinatorics, as for example parabolic double cosets in Coxeter groups~\cite{BKPST17,DS22}. They also appear in Hecke algebras to enumerate maps in non-orientable surfaces~\cite{GJ96}, and have many connections with representation theory~\cite{Fr41}. Many combinatorial problems can also be reformulated in terms of double cosets~\cite{BGR14,LYWL20}.
 
Double cosets in a quotient of the form $H\backslash G/H$ have an inverse, some double cosets being their own inverse. Our first result is a formula for $|\selfinv^G_H|$, the number of self-inverse double cosets in $H\backslash G /H$.

\begin{thm}
	Let $C_1,...,C_r$ be the conjugacy classes of a finite group $G$, and $H$ a subgroup of $G$. We have 
	\begin{equation}
		|\selfinv^G_H|=\frac1{|H|}\sum_{i=1}^r|H\cap C_i|\cdot \Sq(C_i),
	\end{equation} where $\Sq(C_i)$ is the number of square roots of any element of $C_i$.
\end{thm}


Though some enumeration formulas were already known for self-inverse double cosets, the formule above can be more convenient for computations.
As they are needed in Theorem \ref{thm:2} and Formula (\ref{eq:4}) to compute the number of (self-inverse) double cosets, we give formulas for computing $|C_i|$, $\Sq(C_i)$ and $|H\cap C_i|$, in particular when $G$ is a symmetric group or a general linear group, making use of cycle indexes of their subgroups.

We then apply our formulas to various examples, including quotients of the symmetric groups by cyclic or dihedral subgroups, which can be interpreted as cycles or polygons up to symmetry. We also consider higher dimensional analogs, which can be seen as permutations of the vertices of a polytope up to symmetry.
Applying our formulas to double cosets of particular subgroups of the general linear group formed by permutation matrices, we can enumerate matrices in $\GL_n(\F_q)$ up to permutation of rows and columns. We also explain how to compute $|\GL_\lambda(\F_q)\backslash\GL_n(\F_q)/\GL_\mu(\F_q)|$, and conjecture the following:
\begin{cnj}
	For $\lambda,\mu\vdash n$, $|\GL_\lambda(\F_q)\backslash\GL_n(\F_q)/\GL_\mu(\F_q)|$ is a monic polynomial in $q$ with positive integer coefficients.
\end{cnj}
Considering double cosets of parabolic subgroups in the symmetric groups, we found formulas for sums over Kostka numbers, without having to compute the Kostka numbers themselves. These results can also be used to enumerate different kinds of semistandard Young tableaux and matrices, for they are related by the Robinson-Schensted-Knuth correspondence.
We also consider parabolic double cosets of the Coxeter groups of type B, which can be used to enumerate certain kinds of centrally symmetric matrices, or equivalently semistandard Young tableaux invariant under the Schützenberger involution.
\section{Formulas for the enumeration of double cosets}
Let $G$ be a finite group, and $H$ and $K$ two subgroups of $G$. The group $G$ is partitioned into double cosets $H\backslash G/K=\{HgK | g\in G\}$. The group $G$ acts on the cosets $g_iH\in G/H$ by left multiplication, which gives a representation $R_H$ of the group $G$. The representation $R_H$ is the representation of $G$ induced by the trivial representation of $H$, and can be decomposed into irreducible representations $R_i$ with multiplicities $\mu_i^H$ and characters $\chi_i$:
\begin{equation}
	R_H=\bigoplus_{i=1}^r\mu_i^HR_i.
\end{equation}
It is known that~\cite[Exercise 7.77a]{Fr41}
\begin{equation}
	|H\backslash G/K|=\sum_{i=1}^r\mu_i^H\mu_i^K.
	\label{eq:9}
\end{equation}
In $H\backslash G/H$, inverses of elements of a double coset $HgH$ belong to the same double coset $Hg^{-1}H$, so we say $Hg^{-1}H$ is the inverse of $HgH$. The number of self-inverse double cosets in $H\backslash G/H$ is equal to (see \cite{Fr41}):
\begin{equation}
	\selfinv^G_H=\sum_{i=1}^rc_i\mu_i^H,
	\label{eq:1}
\end{equation} where $c_i=1$ if $R_i$ has a symmetric bilinear invariant ($R_i$ is a real representation), $c_i=-1$ if $R_i$ has an alternating bilinear invariant  ($R_i$ is not real but its character is real), and $c_i=0$ otherwise. These coefficients can be expressed as 
\begin{equation}
	c_i=\frac1{|G|}\sum_{g\in G}\chi_i(g^2).
\end{equation}
They also appear when writing the number of square roots of elements of $G$ in terms of characters. Let $\Sq$ be the map that sends each element $g$ of $G$ to its number of square roots $\Sq(g):=|\{x\in G:x^2=g\}|$, we have (see~\cite{Is76}, pages 49--58):
\begin{equation}
	\Sq = \sum_{i=1}^rc_i\chi_i.
	\label{eq:8}
\end{equation}

To prove Formula (\ref{eq:1}), Frame~\cite{Fr41} used the fact that the number of self-inverse double cosets in $H\backslash G/H$ is equal to 
\begin{equation}
	|\selfinv^G_H|=\frac 1{|G|}\sum_{g\in G}\chi(g^2),
	\label{eq:2}
\end{equation} where $\chi(g^2)$ is the trace of $g^2$ seen as a permutation matrix over $G/H$. It turns out that this formula can be written as a sum over $H$ instead of $G$, which can be quicker to compute when the conjugacy classes of $G$ are well understood.

\begin{thm}
	Let $C_1,...,C_r$ be the conjugacy classes of $G$. We have 
	\begin{equation}
		|\selfinv^G_H|=\frac1{|H|}\sum_{i=1}^r|H\cap C_i|\cdot \Sq(C_i),
	\end{equation} where $\Sq(C_i)$ is the number of square roots of any element of $C_i$.
	\label{thm:2}
\end{thm}


\begin{proof}
	We write (\ref{eq:2}) as 
	\begin{align}
		\begin{split}
					|\selfinv^G_H|&=\frac 1{|G|}\sum_{g\in G}|\{K\in G/H:g^2K=K\}|\\
			&=\frac 1{|G|}\sum_{g\in G}\frac 1{|H|}|\{x\in G:xg^2x^{-1}\in  H\}|\\
			& =\frac 1{|G||H|}\sum_{h\in H}\sum_{g\in G}|\{x\in G:xg^2x^{-1}=h\}|.
		\end{split}
	\end{align}
	
	The quantity $|\{x\in G:xg^2x^{-1}=h\}|$ is equal to $\frac {|G|}{|C_i|}$ if $g^2$ and $h$ belong to the same conjugacy class $C_i$ of $G$, and $0$ otherwise. We get 
	\begin{equation}
		\sum_{g\in G}|\{x\in G:xg^2x^{-1}=h\}|=\frac {|G|}{|C_i|}|\{g\in G:g^2\in C_i\}|=|G|\cdot |\{g\in G:g^2=h\}|,
	\end{equation}
	hence 
	\begin{equation}
		|\selfinv^G_H|=\frac1{|H|}\sum_{h\in H}|\{g\in G:g^2=h\}|.
		\label{eq:7}
	\end{equation}
	This equation can then be simplified into a sum over the conjugacy classes $(C_i)_i$ of $G$.
\end{proof}
It is worth noticing that we can also obtain Formula \eqref{eq:7} by proving that self-inverse double cosets contain $|H|$ elements whose squares are in $H$ (some of which are involutions), while other double cosets contain none. It also implies that self-inverse doubles cosets are the double cosets containing elements of $G$ whose squares are in $H$, as well as the double cosets containing involutions.

We can also show directly that Equation \eqref{eq:1} and Theorem \ref{thm:2} are equivalent. By Equation \eqref{eq:8}, we have 
\begin{equation}
	\frac1{|H|}\sum_i|H\cap C_i|\cdot \Sq(C_i)=\frac1{|H|}\sum_i|H\cap C_i|\sum_j c_j\chi_j(C_i)=\frac1{|H|}\sum_jc_j\sum_{h\in H}\chi_j(h).
\end{equation}

Since $\dps \mu_j^H=\frac 1{|H|}\sum_{h\in H}\chi_j(h)$, we have indeed $\dps \frac1{|H|}\sum_i|H\cap C_i|\cdot \Sq(C_i)=\sum_ic_i\mu_i^H$.

We will also need a formula to compute the number of $(H_1,H_2)$-double cosets of $G$, where $H_1,H_2<G$. This can be obtained using Burnside's lemma for the action of $H_1\times H_2$ on $G$:
\begin{equation}
	|H_1\backslash G/H_2| =\frac1{|H_1|\cdot |H_2|}\sum_{(h_1,h_2)\in H_1\times H_2}|\{ g:g = h_1gh_2\}|.
\end{equation}
 To compute the number of elements $g\in G$ fixed by $(h_1,h_2)\in G$, see that $g = h_1gh_2$ only if $h_1$ and $h_2$ belong to the same conjugacy class $C_i$, and in this case the number of $g\in G$ such that $g=h_1gh_2$ is equal to $|G|/|C_i|$. We can then rewrite the number of double cosets as a sum over the conjugacy classes of $G$:
\begin{equation}
	|H_1\backslash G/H_2|=\frac1{|H_1|\cdot |H_2|}\sum_{i}\frac{|G|}{|C_i|}\cdot |H_1\cap C_i|\cdot |H_2\cap C_i|.
	\label{eq:4}
\end{equation}

\section{Computing $|C_i|$, $\Sq(C_i)$, and $|H\cap C_i|$}
In this section we explain how to compute the quantities $|C_i|$, $\Sq(C_i)$, and $|H\cap C_i|$ for certain groups and subgroups, as we need them for Theorem \ref{thm:2} and Formula (\ref{eq:4}). We usually compute the size of a conjugacy class $C_i$ from the size of its centralizer, since it is equal to $|G|/|C_i|$. 

In a conjugacy class, all elements have the same number of square roots, and all their squares belong to the same conjugacy class. Therefore, $\Sq(C_i)$ is equal to the sum of the sizes of conjugacy classes whose squares are in $C_i$, divided by the size of $C_i$:
\begin{equation}
	\Sq(C_i) = \frac1{|C_i|}\sum_{C_j^2=C_i}|C_j|.
	\label{eq:12}
\end{equation}
The values of $|H\cap C_i|$, with $H<G$, will often be given by cycle indices, whose definition depends on the group we consider. These cycle indices are designed so that if $H_1<G_1$ and $H_2<G_2$, the cycle index of $H_1\times H_2$ is the product of the cycle indices of $H_1$ and $H_2$.

\subsection{The case of the symmetric group $\Ss_n$}
Consider $G=\Ss_n$. The conjugacy class of a permutation in $G$ depends only on its cycle type. Let $\lambda=(\lambda_1,...,\lambda_k)$ be a partition of $n$. We denote by $C_\lambda$ the conjugacy class of permutations of cycle type $\lambda$. For all $k\ge 1$, let $m_k$ be the multiplicity of $k$ in $\lambda$ so that $\lambda$ can be written as $[1^{m_1}2^{m_2}...]$. 
Using Equation (\ref{eq:12}) and the fact that the squares of elements of a conjugacy class $C_\lambda$ belong to $C_{\lambda^2}$, where $\lambda^2$ is the partition obtained from $\lambda$ by splitting each of its even parts into halves, we get:
\begin{equation}
	\Sq(C_\lambda)=\begin{cases}
		0 \qquad \text{ if there exists }k\ge 1\text{ such that }m_{2k} \text{ is odd;}\\
		\dps \prod_{k \text{ odd}}\left(\sum_{i=0}^{\lfloor m_k/2\rfloor}\frac{m_k!}{i!(m_k-2i)!}\left(\frac k2\right)^i\right)\prod_{k \ge 1}\frac{m_{2k}!~k^{m_{2k}/2}}{(m_{2k}/2)!} \quad \text{ otherwise.}
	\end{cases}
\end{equation}
We used this formula to compute the sum of the character tables of the symmetric groups~\cite[A082733]{oeis}, without having to compute every character value. 
Since all characters of the symmetric group are characters of real representations, all $c_i$ in Equation (\ref{eq:8}) are equal to $1$, hence $$ \sum_{\lambda,\mu\vdash n}\chi_\mu(C_\lambda) = \sum_{\lambda\vdash n}\Sq(C_\lambda).$$ Since $\Sq(C_\lambda)=\prod_{k\ge 1}\Sq(C_{[k^{m_k}]})$,
its generating function  can be written as 
\begin{align}
	\begin{split}
		\sum_{n\ge 0}x^n\sum_{\lambda,\mu\vdash n}\chi_\mu(C_\lambda) &= \prod_{k\ge 0} \sum_{m\ge 0}x^{km}\Sq(C_{[k^m]})
		\\
		&= \prod_{k\ge 0}\left(\sum_{m\ge 0}x^{4km}\frac{(2m)!k^m}{m!}\right)\prod_{k\ge 0}\left(\sum_{m\ge 0}\sum_{i=0}^{\lfloor m/2\rfloor}\frac{x^{(2k+1)m}m!}{i!(m-2i)!}\left(k+\frac 12\right)^i\right).
	\end{split}
\end{align}
In a recent paper~\cite{ADP24}, Ayyer, Dey and Paul discovered independently similar formulas to compute the sum  of character tables of generalized symmetric groups.\\

For certain subgroups $H<\Ss_n$, we give formulas for $|H\cap C_\lambda|$, which can be written nicely using the cycle index of $H$ acting on $G$ by conjugation, which is 
\begin{equation}
	Z_H(x_1,x_2,...)=\frac 1{|H|}\sum_{h\in H}\prod_{k\ge 1}x_k^{j_k(h)}=\frac 1{|H|}\sum_{\lambda\vdash n}|H\cap C_\lambda|\cdot x_\lambda,
\end{equation}
where $j_k(h)$ is the number of cycles of length $k$ in the permutation of $G$ corresponding to $h$, and $x_\lambda=x_{\lambda_1}\cdots x_{\lambda_r}$ where $\lambda=(\lambda_1,...,\lambda_r)$.
If $H_1<\Ss_n$ and $H_2<\Ss_m$, their product $H_1\times H_2$ is a subgroup of $\Ss_n\times \Ss_m<\Ss_{n+m}$, and its cycle index as a subgroup of $\Ss_{n+m}$ is equal to:
\begin{equation}
	Z_{H_1\times H_2}(x_1,x_2,...) = Z_{H_1}(x_1,x_2,...)\cdot Z_{H_2}(x_1,x_2,...).
\end{equation}

When $H$ is the cyclic group $Z_n$, we have 
\begin{equation}
	\dps Z_{Z_n}(x_1,x_2,...)=\frac 1n\sum_{d|n}\varphi(d)x_d^{n/d},
	\label{eq:13}
\end{equation} where $\varphi$ is the Euler totient function. When $H$ is the dihedral group $D_n$, we have 
\begin{equation}
	Z_{D_n}(x_1,x_2,...)=\begin{cases}\dps \frac 12x_1x_2^{(n-1)/2}+\frac1{2n}\sum_{d|n}\varphi(d)x_d^{n/d} &\text{ if } n \text{ is odd;}  \\
		\dps \frac 14\left(x_1^2x_2^{n/2-1}+x_2^{n/2}\right)+\frac1{2n}\sum_{d|n}\varphi(d)x_d^{n/d}&\text{ if } n \text{ is even.}\end{cases}
	\label{eq:14}
\end{equation}

Let $\lambda=(\lambda_1,...,\lambda_r)\vdash n$. We define as usual $\dps z_\lambda:=\frac{|G|}{|C_\lambda|}=\prod_{k\ge 1}k^{m_k}m_k!$
and $\dps \Ss_\lambda:=\prod_{i=1}^r\Ss_{\lambda_i}$, which is a subgroup of $\Ss_n$, and we have
\begin{equation}
	Z_{\Ss_\lambda}(x_1,x_2,...)=\prod_{i=1}^r\left(\sum_{\nu\vdash \lambda_i}\frac{x_\nu}{z_\nu}\right).
	\label{eq:3}
\end{equation}

\subsection{The case of the general linear group $\GL_n(\F_q)$}
Consider now $G=\GL_n(\F_q)$. Let $q$ be a prime power, $\Phi_q\subset \F_q[X]$ be the set of irreducible monic polynomials over $\F_q$, and $\mathcal P$ the set of integer partitions. Let $\mathcal F_{n,q}$ the set of mappings $f:\Phi_q\to \mathcal P$ such that 
\begin{equation}
	\sum_{\phi\in \Phi_q}\deg(\phi)\cdot |f(\phi)|=n.
\end{equation}

Conjugacy classes in $M_n(\F_q)$  are in bijection  with mappings in $\mathcal F_{n,q}$~\cite[Section IV.2]{Mac98}.
while conjugacy classes in $\GL_n(\F_q)$ correspond to mappings $f$ such that $f(X)$ is the empty partition. To see that, let us consider the $\F_q[X]$-module on $\F_q^n$ in which $X$ acts as a given matrix $M\in M_n(\F_q)$. This module is isomorphic to a direct sum~\cite[ Thm. 12.1.6]{DF91}
\begin{equation}
	\bigoplus_{\phi\in \Phi_q}\bigoplus_{j=1}^{\ell_\phi} \F_q[X]/(\phi^{\lambda_{\phi,j}})
\end{equation}
with $\lambda_{\phi,1}\ge \lambda_{\phi,2}\ge ... \ge \lambda_{\phi,\ell_\phi}$ for all $\phi \in \Phi_q$.
Two matrices are conjugate if and only if they define isomorphic modules, \emph{i.e.}, for each $\phi\in \Phi_q$, the corresponding partitions $(\lambda_{\phi,1},...,\lambda_{\phi,\ell_\phi})$ are the same.

This property is very useful for enumeration in $M_n(\F_q)$ (see~\cite{Mor06} for examples). It is a generalization of Jordan reduction: Jordan blocks $J_{k_1}(\alpha),...,J_{k_r}(\alpha)$ correspond to the polynomial $X-\alpha$ being mapped to the partition $(k_1,...,k_r)$.
We denote by $C_f$ the conjugacy class corresponding to a mapping $f$.

Let $\lambda$ be a partition of $n$, let $m_k$ be its number of parts equal to $k$ and $d_k=m_1+2m_2+...+km_k+k(m_{k+1}+...+m_n)$. We define
\begin{equation}
	c_d(\lambda):=\prod_{k\ge 1}\prod_{i=1}^{m_k}\left(q^{d_k\cdot d}-q^{(d_k-i)\cdot d}\right).
\end{equation}

We then have the equality~\cite[Lemma 2, p. 146]{Ku81}
\begin{equation}
	z_f:=\frac{|G|}{|C_f|}=\prod_{\phi\in \Phi_q}c_{\deg \phi}(f(\phi)).
	\label{eq:10}
\end{equation}

For enumerative purposes, we distinguish conjugacy classes according to their types.
\begin{defi}
	Let $C_f$ be a conjugacy class of $\GL_n(\F_q)$ corresponding to the map $f$ sending each $\phi \in \Phi_q$ to a partition $\lambda_\phi$. The \emph{type} of $C_f$ is the sequence of multisets $$\ctype(C_f):=(\{\lambda_\phi:\deg(\phi)=d\})_{d\ge 1}.$$
\end{defi}
We can see that the size of $C_f$ only depends on its type (\emph{cf.} Equation (\ref{eq:10})). We can also compute the number of conjugacy classes of a given type in $\GL_n(\F_q)$.
\begin{prop}
	Let $\ctype$ the type of a conjugacy class in $\GL_n(\F_q)$, \emph{i.e.}, a sequence of multisets of integer partitions $(\ctype_d)_{d\ge 1}$ such that $\sum_{d\ge 1} d\cdot \sum_{\lambda \in \ctype_d}|\lambda|=n$. For all $d\ge 1$ and $\lambda \in \mathcal P$, let $m_{d,\lambda}$ be the number of partitions equal to $\lambda$ in $\ctype_d$. The number of distinct conjugacy classes of type $\ctype$ in $\GL_n(\F_q)$ is equal to
	\begin{equation}
		\ncl_\ctype=\prod_{d\ge 1}\binom{I_d(q)}{|\ctype_d|}\frac{|\ctype_d|!}{\prod_{\lambda}m_{d,\lambda}!},
		\label{eq:11}
	\end{equation}
	where $I_d(q)$ is the number of monic irreducible polynomials of degree $d$ in $\F_q[X]$ which are distinct from $X$; we have $I_1(q)=q-1$ and $I_d(q)=\frac 1d\sum_{k|d}(-1)^{\mu(d/k)}q^k$ for $d>1$.
\end{prop}
\begin{proof}
	Conjugacy classes of type $\ctype $ in $\GL_n(\F_q)$ are in bijection with mappings $f:\Phi_d\to \mathcal P$ such that $f(X)=\emptyset$, and for all $d\ge 1$ the multiset $\{f(\phi):\phi\in \Phi_d,\deg(\phi) = d\}$ is equal to $\ctype_d$. It is well-known that $|\Phi_d|$ is equal to the necklace polynomial $\frac 1d\sum_{k|d}(-1)^{\mu(d/k)}q^k$.
\end{proof}

There does not seem to be an explicit formula for $\Sq(C_f)$. To compute it, we use the following method:
\begin{enumerate}
	\item Generate all conjugacy classes in $\GL_n(\F_q)$ as mappings $f:\Phi_q\to \mathcal P$;
	\item For each conjugacy class, find the conjugacy class containing the squares of its elements;
	\item $\Sq(C_f)$ is then the sum of the sizes of conjugacy classes whose squares are in $C_f$, divided by the size of $C_f$.
\end{enumerate}

Steps (1) and (3) are straightforward. Concerning step (2), we now explain how to compute the square of a conjugacy class using only the corresponding mapping. In the following, for all $\phi \in \Phi_q$ we define $\phi_2$ as the polynomial whose roots are the squares of the roots of $\phi$, which is equal to $\phi(\sqrt X)\cdot \phi(-\sqrt X)$.
\begin{prop}
	Let $f:\Phi_q\to \mathcal P$ corresponding to the conjugacy class $C_f\subset \GL_n(\F_q)$. Let $f^2$ be the mapping corresponding to $C_f^2$, \emph{i.e.}, the conjugacy class containing the squares of the elements of $C_f$. If $f$ maps $\phi$ to a partition $\lambda=(\lambda_1,...,\lambda_i)$, then $f^2$ maps each irreducible factor of $\phi_2$ to $\lambda$ if $q$ is odd and to $(\lfloor (\lambda_1+1)/2\rfloor,\lfloor \lambda_1/2\rfloor,...,\lfloor (\lambda_i+1)/2\rfloor,\lfloor \lambda_i/2\rfloor)$ if $q$ is even.
	\label{pr:1}
\end{prop}
\begin{proof}
	We consider the Jordan normal form of elements of $C_f$ in the algebraic closure of $\F_q$. If $f$ maps a polynomial $\phi$ to a partition $(\lambda_1,...,\lambda_i)$, the corresponding Jordan form contains the blocks $J_{\lambda_1}(\rho),...,J_{\lambda_i}(\rho)$ for each root of $\phi$. We then need to find the Jordan normal form of the square of a Jordan block. Coefficients of the upper diagonal of $J_k (\rho)^2$ are equal to $2\rho$, which is equal to $0$ if and only if $q$ is even since we are in $\GL_n(\F_q)$ and $\rho\neq 0$. If $q$ is odd, $J_k (\rho)^2$ is similar to $J_k(\rho^2)$ since one can be obtained from the other using elementary operations. If $q$ is even, $J_\lambda(\rho)^2$ has only two nonzero diagonals and it decomposes into two Jordan blocks $J_{\lfloor (k+1)/2\rfloor}(\rho^2)$ and $J_{\lfloor k/2\rfloor}(\rho^2)$:
	
	$$\begin{pmatrix} 
		\rho & 1 &  &  &  &  \\
		& \rho & 1 &  & (0) &  \\
		&  & \ddots & \ddots &  &  \\
		&  &  & \ddots & \ddots &  \\
		& (0) &  &  & \rho & 1 \\
		&  &  &  &  & \rho \\
	\end{pmatrix}^2=\begin{pmatrix} 
		\rho^2 & 2\rho& 1 &  &  & & \\
		& \rho^2 & 2\rho & 1 & & (0) & \\
		&  & \ddots & \ddots & \ddots &  &\\
		&  &  & \ddots & \ddots & \ddots &\\
		&  &  &  & \rho^2 & 2\rho  & 1\\
		& (0) &  &  &  & \rho^2 & 2\rho \\
		&  &  &  &  &  & \rho^2\\
	\end{pmatrix}.\qedhere$$
\end{proof}

Proposition (\ref{pr:1}) allows us to compute $f$ using only $f^2$. For some distinct irreducible polynomials $\phi$ and $\psi$, $\phi_2$ and $\psi_2$ may have common factors. In this case, if their factors are mapped to several partitions, we group their parts into one partition.

\begin{figure}[H]
	$$
	\begin{array}{ccccccccccc}
		f & :& \Phi_2 & \to & \mathcal P&\text{\qquad }&f^2 & :& \Phi_2 & \to & \mathcal P\\
		&&X^2+X+1&\mapsto& (5,2)&&&&X^2+X+1&\mapsto& (3,2,1,1)\\
		&&&&&&&&&&\\
		g & :& \Phi_3 & \to & \mathcal P&&g^2 & :& \Phi_3 & \to & \mathcal P\\
		&&X^2+1&\mapsto& (2,1)&&&&X+1&\mapsto& (2,2,1,1,1)\\
		&&X+1&\mapsto& (1)\\
	\end{array}
	$$
	\caption{\centering Two mappings with their respective squares. The mapping $f$ corresponds to a conjugacy class  in $\GL_{14}(\F_2)$, and $g$ to a conjugacy class  in $\GL_7(\F_3)$.}
\end{figure}

Similarly to the case of the symmetric group, the values of $|H\cap C_f|$ for $H<\GL_n(\F_q)$ and all $C_f$ can be summarized in a cycle index
\begin{equation}
	Z_H((x_{\phi,k})_{\phi\in \Phi_q,k\ge 1})=\frac 1{|H|}\sum_{f}|H\cap C_f|\cdot x_f,
\end{equation}
where $\dps x_f=\prod_{\phi\in \Phi_q}\prod_{k\in f(\phi) }x_{\phi,k}$. If $H_1<\GL_n(\F_q)$ and $H_2<\GL_m(\F_q)$, their product $H_1\times H_2$ is a subgroup of $\GL_n(\F_q)\times \GL_n(\F_q)<\GL_{n+m}(\F_q)$, and its cycle index as a subgroup of $\GL_{n+m}(\F_q)$ is equal to
\begin{equation}
	Z_{H_1\times H_2}(x_{\phi,k}) = Z_{H_1}(x_{\phi,k})\cdot Z_{H_2}(x_{\phi,k}).
\end{equation}

Let $\lambda=(\lambda_1,...,\lambda_r)\vdash n$. We define $\dps \GL_\lambda(\F_q):=\prod_{i=1}^r\GL_{\lambda_i}(\F_q)$, which is a subgroup of $\GL_n(\F_q)$, and we have 
\begin{equation}
	Z_{\GL_\lambda(\F_q)}(x_{\phi,k})=\prod_{i=1}^r\left(\sum_{|f|=\lambda_i}\frac{x_f}{z_f}\right).
\end{equation}

Let $T_s$ be the subgroup of upper triangular matrices in $\GL_n(\F_q)$. According to the Bruhat decomposition, $|T_s\backslash\GL_n(\F_q)/T_s|$ is equal to $n!$ and $|\selfinv^{\GL_n(\F_q)}_{T_s}|$ is the number of involutions on $n$ elements. We could also compute $|T_s\backslash\GL_n(\F_q)/T_s|$ and $|\selfinv^{\GL_n(\F_q)}_{T_s}|$ with Formulas (\ref{eq:4}) and (\ref{thm:2}), using a method to compute $|T_s\cap C_f|$ given by Fuchs and Kirillov~\cite{FK22}. 

\section{Applications}
We will now apply our formulas to various examples, including the enumeration of cycles, polygons and permutations of vertices of higher dimensional polytopes up to symmetry. We consider double cosets of parabolic subgroups of the symmetric groups and type B Coxeter groups, using them to enumerate certain types of matrices and Young tableaux. We finally enumerate certain double cosets of the general linear groups over a finite field.
\subsection{Cycles and polygons}
We consider cycles and polygons that can be drawn by connecting $n$ equally spaced points on a circle $x_1,...,x_n$, cycles being oriented and polygons not (see Figures \ref{fig:inversecycle} and \ref{fig:hexagones}). Let $\sigma\in \Ss_n$, we write $\Cyc(\sigma)$ (resp. $\Poly(\sigma)$) the cycle (resp. polygon) obtained by connecting cyclically the points $x_{\sigma(1)},...,x_{\sigma(n)}$. Permutations give the same cycle (resp. polygon) if and only if they belong to the same coset $\Ss_n/Z_n$ (resp. $\Ss_n/D_n$). 
Cycles of length $n$ up to rotation are in bijection with $Z_n\backslash \Ss_n/Z_n$ and polygons of length $n$ up to rotation and reflection are in bijection with $D_n\backslash \Ss_n/D_n$. An easy way to obtain the inverse of a cycle up to rotation $\Cyc(\sigma)$ (resp. a polygon up to rotation and reflection $\Poly(\sigma)$) is to relabel the vertices $x_1,...,x_n$ by $\sigma(1),...,\sigma(n)$ and then connecting vertices with labels $1,2,...,n$ (see Figure \ref{fig:inversecycle}). Many properties on polygons can be reformulated nicely on their inverse polygons. For example, when considering polygons with an odd number of vertices, polygons without parallel sides have as inverses the polygons such that "middles" of adjacent vertices (\emph{i.e.}, the only vertex at the same distance from both) are all distinct (\emph{cf.}~\cite[A309318] {oeis}). The numbers of symmetries of inverse polygons are also equal, although they do not necessarily share the same type of symmetry.

\begin{figure}[H]
	\centering
	\includegraphics[width=0.5\linewidth]{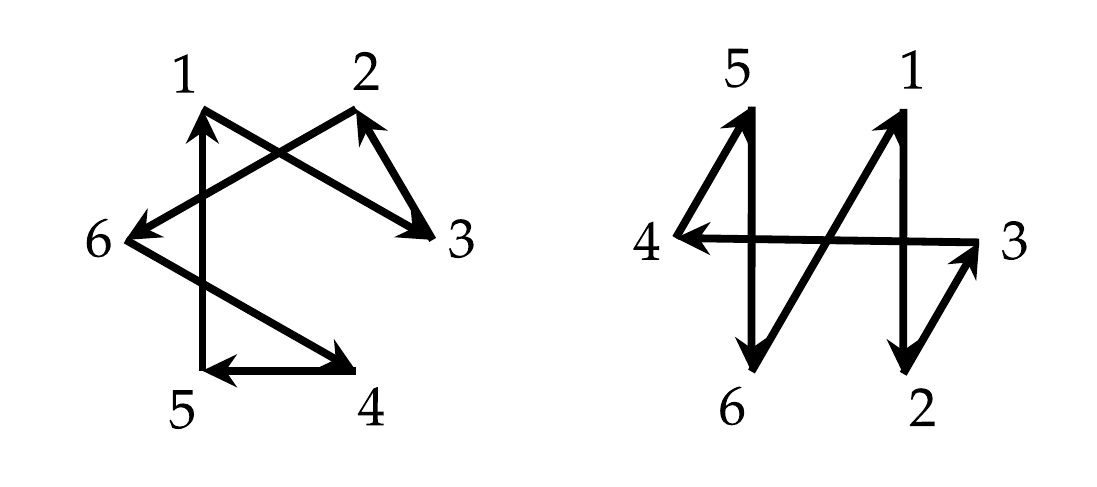}
	\caption{The cycle $\Cyc(513264)$ up to rotation and its inverse in $Z_n\backslash \Ss_n/Z_n$.}
	\label{fig:inversecycle}
\end{figure}

The following results will be given without proof, since they are direct applications of Theorem \ref{thm:2} and Equations (\ref{eq:13}) and (\ref{eq:14}).
\begin{prop}
	The number of self-inverse double cosets in $Z_n\backslash \Ss_n/Z_n$ is
	\begin{equation}
		|\selfinv^{\Ss_n}_{Z_n}|=\frac 1n\sum_{d|n}\varphi(d)\cdot \Sq(C_{[d^{n/d}]}).
	\end{equation}
\end{prop}

\begin{table}[H]
	\centering
	\begin{tabular}{cccccccccccc}
		\toprule
		$n$ & $2$ & $3$ & $4$ & $5$ & $6$ & $7$ & $8$ & $9$ & $10$ & $11$ & $12$\\ \midrule
		$|Z_n\backslash \Ss_n/Z_n|$ & $1$ & $2$ & $3$ & $8$ & $24$ & $108$ & $640$ & $4492$ & $36336$ & $329900$ & $3326788$\\ \midrule
		$|\selfinv^{\Ss_n}_{Z_n}|$ & $1$ & $2$ & $3$ & $6$ & $14$ & $34$ & $98$ & $294$ & $952$ & $3246$ & $11698$\\ \bottomrule
	\end{tabular}
	\caption{First values of $|Z_n\backslash \Ss_n/Z_n|$ and $\selfinv^{\Ss_n}_{Z_n}$.}
\end{table}

\begin{prop}
	The number of self-inverse double cosets in $D_n\backslash \Ss_n/D_n$ is equal to
	\begin{equation}
		|\selfinv^{\Ss_n}_{D_n}|=\begin{cases}\dps \frac 12\Sq(C_{[1,2^{(n-1)/2}]})+\frac 1{2n}\sum_{d|n}\varphi(d)\cdot \Sq(C_{[d^{n/d}]})&\text{ if } n \text{ is odd;}\\
			\dps \frac 14\left(\Sq(C_{[1^2,2^{n/2-1}]})+\Sq(C_{[2^{n/2}]})\right)+\frac 1{2n}\sum_{d|n}\varphi(d)\cdot \Sq(C_{[d^{n/d}]})&\text{ if }n\text{ is even.}\end{cases}
	\end{equation}
\end{prop}

\begin{table}[H]
	\centering
	\begin{tabular}{ccccccccccc}
		\toprule
		$n$ & $3$ & $4$ & $5$ & $6$ & $7$ & $8$ & $9$ & $10$ & $11$ & $12$ \\ \midrule
		$|D_n\backslash \Ss_n/D_n|$ & $1$ & $2$ & $4$ & $12$ & $39$ & $202$ & $1219$ & $9468$ & $83435$ & $836017$\\ \midrule
		$|\selfinv^{\Ss_n}_{D_n}|$ & $1$ & $2$ & $4$ & $8$ & $17$ & $52$ & $153$ & $482$ & $1623$ & $5879$ \\ \bottomrule
	\end{tabular}
	\caption{First values of $|D_n\backslash \Ss_n/D_n|$ and $\selfinv^{\Ss_n}_{D_n}$.}
\end{table}

\begin{figure}[ht]
	\centering
	\includegraphics[width=0.7\linewidth]{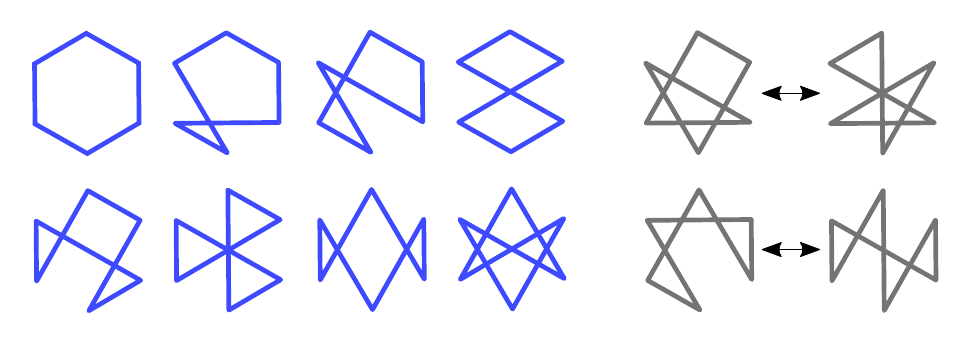}
	\caption{Hexagons up to rotation and reflection. The ones without an arrow are self-inverse.}
\label{fig:hexagones}
\end{figure}

We notice the following property, which follows from the fact that, when $n\equiv 3\mod 4$, reflections in $D_n$ are odd permutations.
\begin{prop}
	We have $|\selfinv^{\Ss_n}_{Z_n}|=2\cdot |\selfinv^{\Ss_n}_{D_n}|$ when $n\equiv 3\mod 4$.
\end{prop}
\begin{proof}
	Each self-inverse double coset in $Z_n\backslash \Ss_n/Z_n$ is contained in a self-inverse double coset in $D_n\backslash \Ss_n/D_n$. Conversely, each double coset $D_ngD_n$ is partitioned into at most $4$ different double cosets in $Z_n\backslash \Ss_n/Z_n$ containing respectively $g$, $sg$, $gs$ and $sgs$, where $s$ is a reflection in $D_n$ ($g$ and $sgs$ may eventually belong to the same double coset in $Z_n\backslash \Ss_n/Z_n$, but not $g$ and $sg$). Consider $g\in \Ss_n$ belonging to a self-inverse double coset in $D_n\backslash \Ss_n/D_n$, \emph{i.e.}, there exists $x,x'\in D_n$ such that $g^{-1}=xgx'$. When $n\equiv 3\mod 4$, reflections in $D_n$ are odd permutations, hence $x$ and $x'$ are either both rotations or both reflections. This implies that $g$ or $gs$ belong to a self-inverse double coset in $Z_n\backslash \Ss_n/Z_n$. We then observe that both $g$ and $gs$ belong to self-inverse double cosets if and only if $g$ and $sgs$ belong to the same double coset, hence each self-inverse double coset in $D_n\backslash \Ss_n/D_n$ corresponds to exactly two self-inverse double cosets in $Z_n\backslash \Ss_n/Z_n$.
\end{proof}

\subsection{Invertible Boolean functions, permutations of vertices of polytopes}
Invertible Boolean functions of $n$ variables, \emph{i.e.}, permutations of $\F_2^n$ up to permutation and complementation of the variables, are in bijection with $\cube_n\backslash \Ss_{2^n}/\cube_n$, where $\cube_n$ is the symmetry group of the cube of dimension $n$~\cite{LYWL20}. The number of such functions~\cite[A000654]{oeis} was computed up to $n=9$ in~\cite{LYWL20}. Using Theorem \ref{thm:2}, we can compute $|\selfinv^{\Ss_{2^n}}_{\cube_n}|$, which is the number of involutive Boolean functions of $n$ variables up to permutation and complementation of the variables, with the first values given in Table \ref{tab:3}.

\begin{table}[H]
	\centering
	\begin{tabular}{cccccc}
		\toprule
		$n$ & $1$& $2$ &$3$ & $4$ & $5$\\ \midrule 
		$|\cube_n\backslash \Ss_{2^n}/\cube_n|$ & $1$ & $2$ & $52$ & $142090700$ & $17844701940501123640681816160$ \\ \midrule
		$|\selfinv^{\Ss_{2^n}}_{\cube_n}|$ & $1$ & $2$ & $24$ & $120930$ & $5854443002344516$ \\ \bottomrule
	\end{tabular}
	\caption{First values of $|\cube_n\backslash \Ss_{2^n}/\cube_n|$ and $|\selfinv^{\Ss_{2^n}}_{\cube_n}|$.}
	\label{tab:3}
\end{table}

Elements of $\cube_n\backslash \Ss_{2^n}/\cube_n$ can also be seen as permutations of the vertices of an $n$-cube, up to isometry (\emph{cf.} Figure \ref{fig:cubes}).

\begin{figure}
	\centering
	\includegraphics[width=0.7\linewidth]{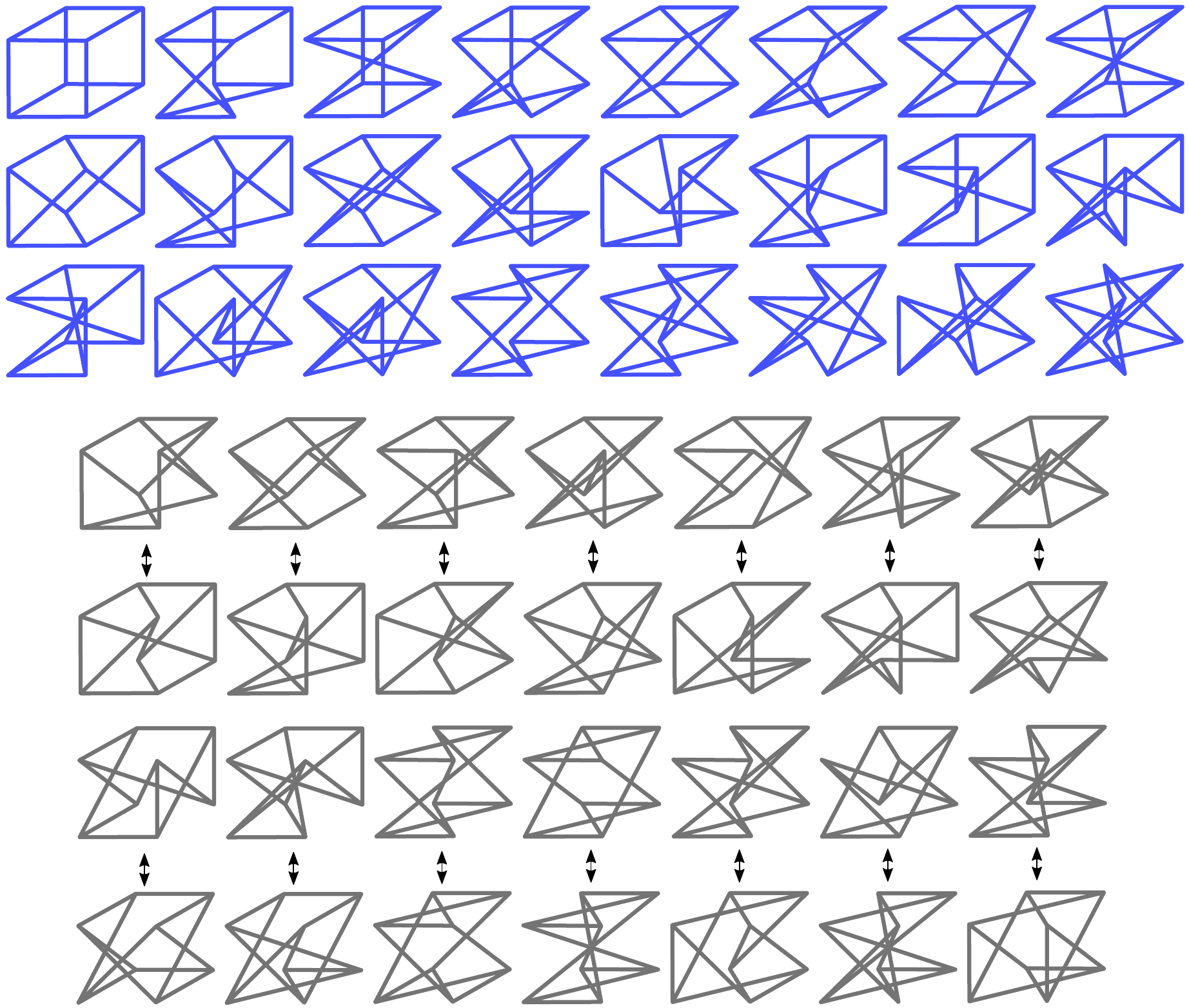}
	\caption{Permutations of the vertices of a cube, up to isometry. The ones without an arrow are self-inverse.}
\label{fig:cubes}
\end{figure}

Let $\octa_n$ be the symmetry group of the hyperoctahedron of dimension $n$ ($\octa_n$ is isomorphic to $\cube_n$, but they act on different sets). Double cosets in $\octa_n\backslash \Ss_{2n}/\octa_n$ are in bijection with integer partitions of $n$~\cite{DS22}, and intervene in the study of zonal polynomials~\cite{Mac98}. In~\cite{DS22}, Diaconis and Simper explained how to compute to which double coset a permutation $\sigma \in \Ss_{2n}$ belongs, by constructing a graph $T(\sigma)$ on vertices $\{1,2,...,2n\}$ such that for all $1\le i\le n$, an edge connects vertices $2i-1$ and $2i$, and another edge connects vertices $\sigma(2i-1)$ and $\sigma(2i)$. Let $\lambda$ be the partition of $n$ obtained by dividing the lengths of the cycles of $T(\sigma)$ by 2, then $\sigma$ belongs to the double coset indexed by $\lambda$, which we denote by $\octa_\lambda$.

Similarly to $ \Ss_{2^n}/\cube_n$, cosets of $\octa_n$ in $\Ss_{2n}$ correspond to permutations of the vertices of the hyperoctahedron. These can also be seen as perfect matchings of the complete graph $K_{2n}$ by projecting the line segments between non-adjacent vertices on a Coxeter plane.

We show that every double coset in $\octa_n\backslash \Ss_{2n}/\octa_n$ contains at least one involution. More precisely, we have the following result.
\begin{prop}
	Let $\lambda\vdash n$ with multiplicities $(m_k)_{k\ge 1}$ and $\octa_\lambda$ the double coset of $\octa_n\backslash \Ss_{2n}/\octa_n$ corresponding to $\lambda$. The number of involutions in $\octa_\lambda$ is equal to
	\begin{equation}
		\frac{n!}{\dps \prod_{k\ge 1}m_k!k!^{m_k}}\prod_{k\ge 1}\left(\sum_{i=0}^{\lfloor m_k/2\rfloor}a_k^ib_k^{m_k-2i}\frac{m_k!}{2^ii!(m_k-2i)!}\right)
	\end{equation}
   where $\dps a_k:=2^{2k-1}k!(k-1)!=|\octa_{(k)}|$ and $b_k:=\begin{cases}
   	2^{k-1}k! &\text{ if }k \text{ is even;}\\
   	2^{k-1}(k-1)!(k+1) &\text{ if }k \text{ is odd.}
   \end{cases}$ is the number of involutions contained in $\octa_{(k)}$.
   
\end{prop}
\begin{proof}
	This statement follows from the fact that $b_k$ is indeed the size of the set $I_k$ of involutions of $\octa_{(k)}$. To prove it, we observe that for $k>1$, the group generated by transpositions $(1,2),(3,4)...(2k-1,2k)$ acts faithfully by conjugation on $I_k$, and the number of fixed points of involutions in $I_k$ (which is either $0$ or $2$) is fixed by their action. For all $\sigma\in I_k$, the orbit of this action containing $\sigma$ is characterized by the graph $U(\sigma)$ obtained by relabeling each vertex $i\in T(\sigma)$ by $(\lceil i/2\rceil,\lceil \sigma(i)/2\rceil)$. The graphs $\{U(\sigma):\sigma\in I_k\}$ are the cycles of length $2k$ with vertices in $\{1,...,k\}^2$, symmetric with respect to the permutation of both coordinates, and such that for all $i\in \{1,...,k\}$, there is an edge connecting two vertices whose first coordinates are equal to $i$ (see Figure \ref{fig:3}). These graphs have either $0$ or $2$ fixed points, \emph{i.e.}, vertices whose both coordinates are equal. There are $k!/2$ such graphs having $2$ fixed points, and the number of those with no fixed point is equal to $(k-1)!/2$ if $k$ is odd and $0$ otherwise. Since every orbit contains $2^k$ elements, we get the value of $b_k$.
\end{proof}
Since all double cosets in $\octa_n\backslash \Ss_{2n}/\octa_n$ contain an involution, they are all self-inverse. This property also follows from the fact that $(\Ss_{2n},\octa_n)$ is a Gelfand pair, meaning the permutation representation of $\Ss_{2n}$ on the cosets of $\octa_n$ is multiplicity-free, hence $|\selfinv^{\Ss_{2n}}_{\octa_n}|=|\octa_n\backslash \Ss_{2n}/\octa_n|$ by Equations (\ref{eq:9}), (\ref{eq:1}) and the fact that all characters of symmetric groups are characters of real representations.

\begin{figure}[ht]
	\centering
	\includegraphics[width=0.7\linewidth]{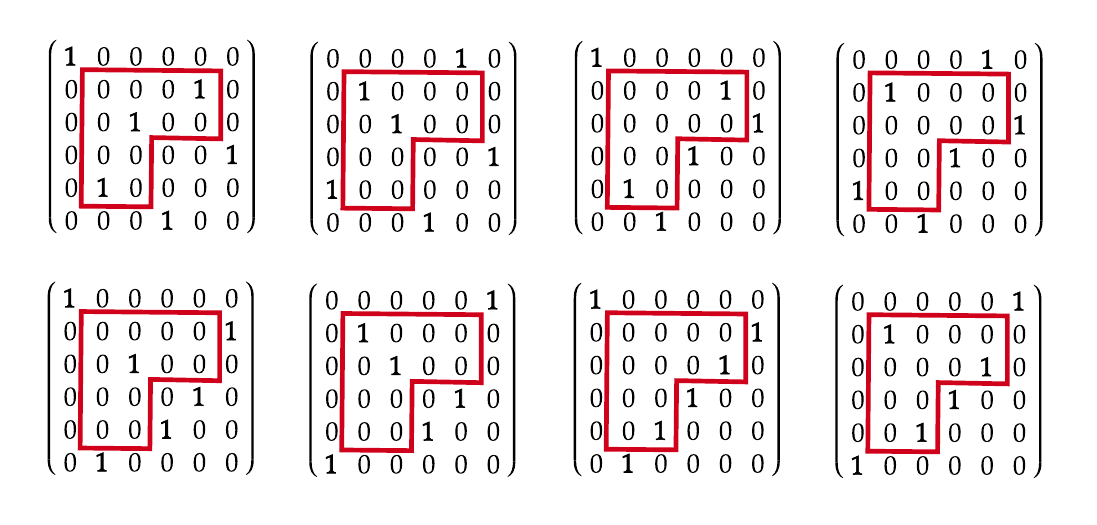}
	\caption{The $8$ involutions corresponding to a cycle of length $6$.}
\label{fig:3}
\end{figure}

We have seen that we can compute the number of permutations up to isometry of the vertices of the hypercube and of the hyperoctahedron. We can also compute the number of permutations up to isometry of the vertices of some other polytopes:
\begin{itemize}
	\item Icosahedron: \numprint{34288}, out of which \numprint{1210} are self-inverse;
	\item Dodecahedron: \numprint{168951598306448}, out of which \numprint{197995308} are self-inverse;
	\item 24-cell: \numprint{467520560895051562}, out of which \numprint{15184383396} are self-inverse;
	\item 600-cell: \numprint{32260334266247719220621711294803108539509772054099294995638847185096612448133502070635002726060994969511867690382633171943631821094864835171510244658706237050954969326843751674806593463033600},\\ out of which \numprint{1107274514227682955989628044559513224247430148585107198991302040384188121361044957390472571582634488} are self-inverse;
	\item 120-cell: its number of permutations up to isometry and the number of those which are self-inverse have respectively 1400 and 710 digits (\emph{cf.} appendix).
\end{itemize}

\subsection{Parabolic double cosets in symmetric groups}

We will now study double cosets in $\Ss_\lambda\backslash \Ss_n/\Ss_\mu$ with $\lambda,\mu\vdash n$ and $\Ss_\lambda,\Ss_\mu$ as defined in Equation \eqref{eq:3}. These double cosets are in bijection with "contingency matrices", \emph{i.e.}, matrices with nonnegative integer coefficients, with row sums $\lambda$ and column sums $\mu$~\cite{DS22}. Inverse double cosets correspond to transposed matrices, thus self-inverse cosets correspond to symmetric matrices.

Using the Robinson-Schensted-Knuth correspondence (see~\cite{St99}, page 316), we know that nonnegative integer matrices with row sums $\lambda$ and column sums $\mu$ are in bijection with pairs of semistandard Young tableaux $(P,Q)$ of the same shape, $P$ having weight $\lambda$ and $Q$ having weight $\mu$; symmetric matrices correspond to pairs where both tableaux are equal.
We then obtain the following formulas involving Kostka numbers for $|\Ss_\lambda\backslash \Ss_n/\Ss_\mu|$ and $|\selfinv^{\Ss_n}_{\Ss_\lambda}|$, which can also be expressed using Theorem \ref{thm:2} and Equation \eqref{eq:4}:
\begin{equation}
	|\Ss_\lambda\backslash \Ss_n/\Ss_\mu|=\sum_{\nu\vdash n}K_{\nu,\lambda}K_{\nu,\mu}=\frac 1{|\Ss_\lambda|\cdot|\Ss_\mu|}\sum_{\nu\vdash n}z_\nu|\Ss_\lambda\cap C_\nu|\cdot |\Ss_\mu\cap C_\nu|,
\end{equation}
\begin{equation}
	|\selfinv^{\Ss_n}_{\Ss_\lambda}|=\sum_{\nu\vdash n}K_{\nu,\lambda}=\frac 1{|\Ss_\lambda|}\sum_{\nu\vdash n}|\Ss_\lambda\cap C_\nu|\cdot \Sq(C_\nu).
\end{equation}

Using Equation \eqref{eq:3}, we get

\begin{equation}
	\sum_{n\ge 0}\sum_{\nu\vdash n}x_\nu\sum_{\lambda\vdash n}\frac{|\Ss_\lambda \cap C_\nu|}{|\Ss_\lambda|}=\prod_{k\ge 1}\frac 1{1-\sum_{\nu \vdash k}x_\nu z_\nu^{-1}}.
	\label{eq:5}
\end{equation}
This allows us to compute new terms of the sequences $(\sum_{\lambda,\mu\vdash n} |\Ss_\lambda\backslash \Ss_n/\Ss_\mu|)_{n\ge 1}$~\cite[A321652]{oeis} and $(\sum_{\lambda \vdash n}|\selfinv^{\Ss_n}_{\Ss_\lambda}|)_{n\ge 1}$~\cite[A178718]{oeis}, without having to compute Kostka numbers. The first sequence gives the number of nonnegative integer matrices with sum $n$, no zero rows and columns, and non-increasing row sums and column sums. The second one gives the number of such matrices which are symmetric.
\begin{table}[H]
	\centering
	\begin{tabular}{ccccccccccc}
		\toprule
		$n$ & $1$& $2$ &$3$ & $4$ & $5$ & $6$ & $7$ & $8$ & $9$&$10$\\ \midrule
		$\sum_{\lambda,\mu\vdash n} |\Ss_\lambda\backslash \Ss_n/\Ss_\mu|$ & $1$ & $5$ & $19$ & $107$ & $573$ & $4050$ & $29093$ & $249301$ & $2271020$ & $23378901$\\ \midrule
		$\sum_{\lambda \vdash n}|\selfinv^{\Ss_n}_{\Ss_\lambda}|$ & $1$ & $3$ & $7$ & $21$ & $57$ & $182$ & $565$ & $1931$ & $6670$ & $24537$\\ \bottomrule
	\end{tabular}
	\caption{First values of A321652 and A178718 in~\cite{oeis}.}
\end{table}
To compute these sequences, we get the values of 
\begin{equation}
	\alpha_\nu:=\sum_{\lambda\vdash n}\frac{|\Ss_\lambda \cap C_\nu|}{|\Ss_\lambda|}
\end{equation}
as the coefficient of $x_\nu$ in $\dps \prod_{k\ge 1}\frac 1{1-\sum_{\nu \vdash k}x_\nu z_\nu^{-1}}$. The values of A178718 are then given by
\begin{equation}
	\sum_{\lambda,\nu\vdash n}K_{\nu,\lambda}=\sum_{\nu \vdash n}\alpha_\nu\cdot \Sq(C_\nu),
\end{equation}
and the values of A321652 by
\begin{equation}
	\sum_{\lambda,\mu,\nu\vdash n}K_{\nu,\lambda}K_{\nu,\mu}=\sum_{\nu \vdash n}\alpha_\nu^2\cdot z_\nu.
\end{equation}
These equations could also be written using symmetric functions. By setting $x_\nu=p_\nu$, we have $|\Ss_\lambda\backslash \Ss_n/\Ss_\mu|=\langle h_\lambda,h_\mu\rangle$, $\dps \frac{|\Ss_\lambda \cap C_\nu|}{|\Ss_\lambda|}=\langle h_\lambda,p_\nu^*\rangle$, and Equation \eqref{eq:5} gives two different expressions of $\dps \sum_{n\ge 0}\sum_{\lambda\vdash n}h_\lambda$.

A similar formula gives the number of $(0,1)$-matrices with sum $n$, no zero row nor column, and non-increasing row sums and column sums~\cite[A068313]{oeis}. By a variant of the Robinson-Schensted-Knuth correspondence (see~\cite{St99}, page 331), $(0,1)$-matrices are in bijection with pairs of semistandard Young tableaux with conjugate shapes. This gives us another formula involving Kostka numbers:

\begin{equation}
	\sum_{\nu\vdash n}K_{\nu,\lambda}K_{\overline \nu,\mu}=\frac 1{|\Ss_\lambda|\cdot|\Ss_\mu|}\sum_{\nu\vdash n}\varepsilon(\nu)z_\nu|\Ss_\lambda\cap C_\nu|\cdot |\Ss_\mu\cap C_\nu|,
	\label{eq:6}
\end{equation}
where $\dps \varepsilon(\nu)=(-1)^{n-l(\nu)}$ is the signature of the permutations belonging to $C_\nu$.
In terms of symmetric functions, Equation \eqref{eq:6} gives two expressions of $\langle e_\lambda,e_\mu\rangle$. We then obtain the values of A068313 as: 
\begin{equation}
	\sum_{\lambda,\mu,\nu\vdash n}K_{\nu,\lambda}K_{\overline\nu,\mu}=\sum_{\nu \vdash n}\alpha_\nu^2\cdot \varepsilon(\nu)z_\nu.
\end{equation}

\subsection{Parabolic double cosets in type B}
We now consider the Coxeter group $B_n$, defined here as the subgroup of $\Ss_{2n}$ consisting in the centrally symmetric permutation matrices of size $2n$. It is generated by $S={s_1,s_2,...,s_n}$, where $s_1,s_2,...,s_n$ are respectively the matrices of the permutations $(1,2)(2n-1,2n)$, $(2,3)(2n-2,2n-1),...,(n,n+1)$.
Let $I,J$ be subsets of $S$. Let $M\in B_n$, we define $_IM_J$ as the matrix obtained by merging rows of $M$ with indices $i$ and $i+1$, and those with indices $2n-i$ and $2n-i+1$ if and only if $s_i\in I$, and by doing the same operation on the columns for each element of $J$. 
Let $W_I,W_J$ be the subgroups of $B_n$ generated respectively by $I$ and $J$. Two matrices $M,M'\in B_n$ belong to the same double coset in $W_I\backslash B_n/W_J$ if and only if $_IM_J= \!_IM'_J$. Similarly to the case of the symmetric group, the contingency matrices obtained this way can be easily described. Let $\lambda_I$ (resp. $\lambda_J$) be the partition obtained by merging the parts with indices $i$ and $i+1$, and those with indices $2n-i$ and $2n-i+1$, in the partition $[1^{2n}]$ for each $s_i\in I$ (resp. $J$). The set $\{_IM_J:M\in B_n\}$ contains all centrally symmetric matrices with row sums $\lambda_I$ and column sums $\lambda_J$. By the Robinson-Schensted-Knuth correspondence, these matrices are in bijection with pairs of semistandard Young tableaux $(P,Q)$ of the same shape, having respective weights $\lambda_I$ and $\lambda_J$, and invariant by the Schützenberger involution.

Conjugacy classes of $B_n$ are in bijection with ordered pairs of integer partitions $(\lambda,\mu)$ such that $|\lambda|+|\mu| = n$. Elements of $B_n$ can be seen as signed permutations, \emph{i.e.}, elements of $\{1,-1\}^n\times \Ss_n$; $((a_1,...,a_n),\sigma)$ belongs to the conjugacy class corresponding to $(\lambda,\mu)$, denoted by $C_{\lambda,\mu}$, if $\lambda$ (resp. $\mu$) contains the lengths of the cycles $(i_1,...,i_k)$ of $\sigma$ such that $a_{i_1}a_{i_2}\cdots a_{i_k}=1$ (resp. $-1$).

The centralizer of an element of $C_{\lambda,\mu}$ has size 
\begin{equation}
	\frac{|B_n|}{|C_{\lambda,\mu}|}=z_\lambda z_\mu2^{\ell(\lambda)+\ell(\mu)}.
\end{equation}
\begin{prop}
	Let $(m_k)_{k\ge 1}$ and $(\overline m_k)_{k\ge 1}$ be the multiplicities of $\lambda$ and $\mu$ respectively. We then have
	\begin{equation}
		\Sq(C_{\lambda,\mu})=\begin{cases}
			0 \qquad \text{ if there exists }k\ge 1\text{ such that }m_{2k} \text{ or }\overline m_k\text{ is odd;}\\
			\dps \prod_{k \text{ odd}}\left(\sum_{i=0}^{\lfloor m_k/2\rfloor}\frac{m_k!k^i2^{m_k-2i}}{i!(m_k-2i)!}\right)\prod_{k \text{ even}}\frac{m_k!~k^{m_k/2}}{(m_k/2)!}\prod_{k\ge 1}\frac{\overline m_k!k^{\overline m_k/2}}{(\overline m_k/2)!} \quad \text{ otherwise.}
		\end{cases}
	\end{equation}
\end{prop}
\begin{proof}
	It is straightforward to check that in a signed permutation, a cycle gives one or two cycles when squared:
	\begin{itemize}
		\item a positive cycle of odd length $k$ gives a positive cycle of length $k$;
		\item a negative cycle of odd length $k$ gives a positive cycle of length $k$;
		\item a positive cycle of even length $k$ gives two positive cycles of length $k/2$;
		\item a negative cycle of even length $k$ gives two negative cycles of length $k/2$.
	\end{itemize}
We then enumerate the signed permutations whose squares contain positive cycles of lengths $\lambda$ and negative cycles of lengths $\mu$.
\end{proof}

Similarly to the case of the symmetric groups, we define a cycle index for subgroups $H$ of $B_n$:
\begin{equation}
	Z_H(x_1,x_2,...;\overline x_1,\overline x_2,...):=\frac 1{|H|}\sum_{h\in H}\prod_{k\ge 1}x_k^{j_k(h)}\overline x_k^{\overline j_k(h)}=\frac 1{|H|}\sum_{|\lambda|+|\mu|=n}|H\cap C_{\lambda,\mu}|\cdot x_\lambda \overline x_\mu,
\end{equation}
where $j_k(h)$ (resp. $\overline j_k(h)$) is the number of positive (resp. negative) cycles of length $k$ in the signed permutation corresponding to $h$.
\begin{prop}
	Let $I\subset S$, let $\ell_1,...,\ell_r$ be the lengths of maximal sequences of consecutive elements of $I$ which do not contain $s_n$, and let $k$ be the length of the maximal sequence of consecutive elements of $I$ which contains $s_n$. Let $\kappa = (\ell_1+1,...,\ell_r+1,\underbrace{1,...,1}_{n-k-\sum_i\ell_i-r})$ (see Figure \ref{fig:2}). The subgroup $W_I$ is isomorphic to $\dps B_k\times \prod_i\Ss_{\kappa_i}$, and its cycle index is equal to
	\begin{equation}
		Z_{W_I}(x_1,x_2,...;\overline x_1,\overline x_2,...)=\sum_{|\lambda|+|\mu|=k}\frac{x_\lambda\overline x_\mu}{z_\lambda z_\mu2^{l(\lambda)+l(\mu)}}\prod_{i=1}^r\left(\sum_{\nu\vdash \kappa_i}\frac{x_\nu}{z_\nu}\right).
	\end{equation} \label{pr:2}
\end{prop}
\begin{proof}
	The subgroup of $B_n$ generated by $\{s_i,s_{i+1},...,s_{i+\ell}\}$ contains matrices of permutations acting on $\{i,i+1,...,i+l,2n-i-\ell+1,...,2n-i,2n-i+1\}$. It is isomorphic to $B_{\ell+1}$ if $I$ contains $s_n$, and to $\Ss_{\ell+2}$ if not. Subsets of consecutive elements of $I$ generate subgroups of $B_n$ acting on disjoint subsets of $\{1,...,2n\}$, which are isomorphic to $B_k, \Ss_{\kappa_1},\Ss_{\kappa_2}...$. Therefore, $W_I$ is isomorphic to their product and the cycle index of $W_I$ is equal the product of the cycle indices of these subgroups.
\end{proof}
\begin{figure}[ht]
	\centering
	\includegraphics[width=0.6\linewidth]{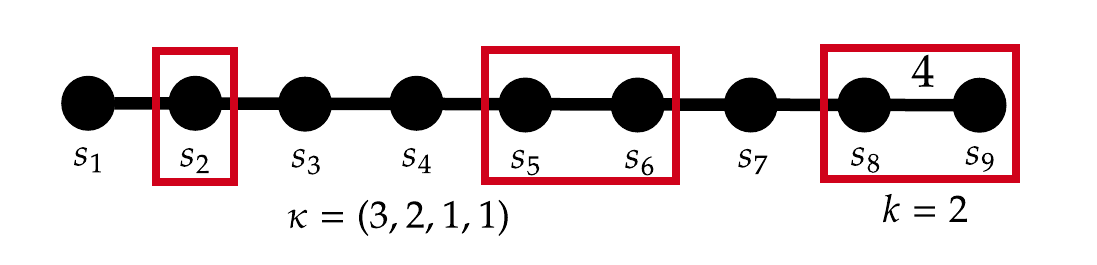}
	\caption{With $n=9$ and $I=\{s_2,s_5,s_6,s_8,s_9\}$, we get $k=2$ and $\kappa = (3,2,1,1)$.}
	\label{fig:2}
\end{figure}

We can use Proposition \ref{pr:2} to obtain the values of $|W_I\cap C_{\lambda,\mu}|$, and efficiently compute $|W_I\backslash B_n/W_J|$ and $\selfinv^{B_n}_{W_I}$.
For example, by taking $n=42$ and
\begin{equation*}
	\begin{split}
		I=\{s_3,s_4,s_4,s_6,s_7,s_8,s_9,s_{12},s_{13},s_{15},s_{17},s_{19},s_{22},s_{23}, s_{24},s_{26},s_{28},s_{29},s_{32},s_{34},s_{35},s_{39},s_{40},s_{41},s_{42}\},
	\end{split}
\end{equation*}
we computed in less than a second 
\begin{align*}
	|W_I\backslash B_n/W_I| &= \numprint{12651530609717357371835019916458406463078400} \text{ and}\\
	|\selfinv_{W_I}^{B_n}|&= \numprint{3280075848815058449211136}.
\end{align*}

\subsection{Matrices up to permutation of rows and columns}

Let $P_n$ be the subgroup of permutation matrices in $\GL_n(\F_q)$. Equation \eqref{eq:4} and Theorem \ref{thm:2} yield the following formulas:
\begin{equation}
	|P_n\backslash \GL_n(\F_q) /P_n|=\frac1{n!^2}\sum_f z_f\cdot |P_n\cap C_f|^2.
\end{equation}
\begin{equation}
	|\selfinv^{\GL_n(\F_2)}_{P_n}|=\frac1{n!}\sum_f |P_n\cap C_f|\cdot \Sq(C_f).
\end{equation}
To compute these quantities, we need to know the values of $|P_n\cap C_f|$, which can be obtained with the following result.
\begin{prop}
	The generating function of all cycle indexes of the subgroups of permutation matrices $P_n$ is equal to:
	\begin{equation}
		\dps \sum_{n\ge 0}t^n Z_{P_n}(x_{\phi,k})=\exp\left(\sum_{k\ge 1}\frac {t^k}k\prod_{X^k-1=\prod_i\phi_{k,i}^{\alpha_{k,i}}}x_{\phi_{k,i},\alpha_{k,i}}\right).
	\end{equation}
\end{prop}
\begin{proof}
	Two permutation matrices $P_\sigma$ and $P_\tau$ are conjugate in $\GL_n(\F_q)$ if and only if $\sigma$ and $\tau$ are conjugate in $\Ss_n$, \emph{i.e.}, they have the same cycle type $\lambda$. More precisely, if $X^k-1$ has irreducible factors $\phi_{k,i}$ with multiplicities $\alpha_{k,i}$, the conjugacy class containg permutation matrices of cycles of length $k$ corresponds to the function mapping each $\phi_{k,i}$ to the partition $(\alpha_{k,i})$ of size $1$. We can then obtain the functions corresponding to each conjugacy class of permutation matrices by combining the functions corresponding to cycles. 
\end{proof}

When $q=2$, we find the following values~\cite[A224879]{oeis}, along with the first values of $|\selfinv^{\GL_n(\F_2)}_{P_n}|$:
\begin{table}[H]
	\centering
	\begin{tabular}{cccccccccc}
		\toprule
		$n$ & $1$& $2$ &$3$ & $4$ & $5$ & $6$ & $7$ & $8$ & $9$\\ \midrule
		$|P_n\backslash \GL_n(\F_2)/P_n|$ & $1$ & $2$ & $7$ & $51$ & $885$ & $44206$ & $6843555$ & $3373513302$ & $5366987461839$\\ \midrule
		$|\selfinv^{\GL_n(\F_2)}_{P_n}|$ & $1$ & $2$ & $5$ & $19$ & $87$ & $706$ & $8309$ & $192090$ & $6961741$ \\ \bottomrule
	\end{tabular}
	\caption{First values of $|P_n\backslash \GL_n(\F_2)/P_n|$ and $|\selfinv^{\GL_n(\F_2)}_{P_n}|$.}
\end{table}
\subsection{Matrices up to scalar multiplication of rows and columns}

\begin{thm}
	Let $\diag_n$ be the subgroup of diagonal matrices in $\GL_n(\F_q)$ and $\lambda=[1^{m_1},2^{m_2}...]$. We have $$  |\diag_n\backslash \GL_n(\F_q)/\diag_n|=\frac {n!^2}{(q-1)^{2n}}\sum_{\lambda\vdash n} \prod_{i=1}^{l(\lambda)}(q-i)\prod_{k\ge 1}\frac{\prod_{j=0}^{k-1}(q^k-q^j)^{m_k}}{k!^{2m_k}m_k!}.$$
	\label{thm:1}
\end{thm}
\begin{proof}
	Matrices of $\diag_n$ whose eigenvalues have multiplicities $\lambda$ belong to the $\dps \frac{\prod_{i=1}^{l(\lambda)}(q-i)}{\prod_{k\ge 1}m_k!}$ different conjugacy classes of type $\ctype$, where $\ctype_1$ is a multiset with multiplicities $\lambda$ and $\ctype_d=\varnothing$ for $d>1$. Each of these classes has a centralizer of cardinality $\dps \prod_i\prod_{j=0}^{\lambda_i-1}(q^{\lambda_i}-q^j)$, and we then apply Equation \eqref{eq:4}.
\end{proof}

As $\diag_n=\GL_{[1^n]}(\F_q)$, Theorem \ref{thm:1} can be generalized to obtain the following result.
\begin{thm}
	Let $\lambda, \mu$ be partitions of $n$. We have
	\begin{equation}
		|\GL_\lambda(\F_q)\backslash\GL_n(\F_q)/\GL_\mu(\F_q)|=\frac 1{|\GL_\lambda(\F_q)|\cdot |\GL_\mu(\F_q)|}\sum_{\ctype}\ncl_\ctype \cdot  a_{\ctype,\lambda}\cdot a_{\ctype,\mu}\prod_{d\ge 1}\prod_{\lambda\in \ctype_d}c_d(\lambda),
	\end{equation}
	where the sum is over all distinct types of conjugacy classes in $\GL_n(\F_q)$, and $a_{\ctype,\lambda}$ is the number of elements of $\GL_{\lambda}(\F_q)$ in any conjugacy class of type $\ctype$.
\end{thm}

\begin{proof}
	We use the fact that, for each conjugacy class  in $\GL_n(\F_q)$, the number of elements of $\GL_\lambda(\F_q)$ belonging to it only depends on its type. To see this, let $\ctype_1,\ctype_2...$ be types of conjugacy classes in $\GL_{\lambda_1}, \GL_{\lambda_2},...$. For all choices $C_{f_1},C_{f_2},...$ of conjugacy classes of type $\ctype_1,\ctype_2,...$, we can compute the number of elements in $C_{f_1}\times C_{f_2}\times\cdots$ of type $\ctype$ by looking at the different ways the supports of $f_1,f_2...$ can intersect. This quantity only depends on the types $\ctype_1,\ctype_2,...$, implying that $a_{\ctype,\lambda}$ is indeed the number of elements of $\GL_{\lambda}(\F_q)$ in any conjugacy class of type $\ctype$. 
	We then combine Equations \eqref{eq:10}, \eqref{eq:11} and \eqref{eq:4} to get the result.
\end{proof}

Using this idea, we can compute for each type $\ctype$ of conjugacy class in $\GL_n(\F_q)$ the number of elements of $\GL_\lambda(\F_q)$ belonging to a conjugacy class of type $\ctype$.
We implemented an algorithm in Sagemath to compute the $a_{\ctype,\lambda}$ and $|\GL_\lambda(\F_q)\backslash\GL_n(\F_q)/\GL_\mu(\F_q)|$.\\

$|\diag_n\backslash \GL_n(\F_q)/\diag_n|$ appears to be a monic polynomial in $q$ with positive integer coefficients and of degree $(n-1)^2$. More generally, we observed the following property, which we checked up to $n=8$.
\begin{cnj}
	For $\lambda,\mu\vdash n$, $|\GL_\lambda(\F_q)\backslash\GL_n(\F_q)/\GL_\mu(\F_q)|$ is a monic polynomial in $q$ with positive integer coefficients.
	\label{cnj:1}
\end{cnj}
For example, we have
\begin{equation}
	|\GL_{(3,2,1)}(\F_q)\backslash\GL_6(\F_q)/\GL_{(4,2)}(\F_q)|=q^4 + 7q^3 + 32q^2 + 89q + 117.
\end{equation}

Conjecture \ref{cnj:1} would imply that, if $\lambda,\mu$ and $\nu$ are partitions of $n$ with $\lambda$ and $\mu$ more refined than $\nu$, $|\GL_\lambda(\F_q)\backslash\GL_\nu(\F_q)/\GL_\mu(\F_q)|$ would also be a monic polynomial in $q$ with positive integer coefficients.

\section*{Acknowledgement}

This research was driven by computer exploration and computations using the open-source
mathematical software \texttt{Sage}~\cite{sage} and its algebraic
combinatorics features developed by the \texttt{Sage-Combinat}
community~\cite{Sage-Combinat}. I would like to thank Jean-Christophe Novelli and Wenjie Fang for their precious help and for proofreading this work.

\bibliography{Double-cosets}{}
\bibliographystyle{plain}

\appendix
\section{}
The number of permutations of the 120-cell up to isometry is equal to\bigskip \\ 
\numprint{61032615558710973309310755426074601165467804119765918056461366022994506284468321941636780457270308038026661519087867596466246642468001395788942060079351234115949978992216861728996076814875692555826408570703950518132165510502598056129619452178429720022875014767262701506915134219969051637826831230175793194787741127385696057817866732554958440915704044247611361518355591978137170814452800960932948683266322770696000132930228740441834367046689461996370406273766426655627766297455414369782083660074361285376742747916471638023533627762620149291669671007220148012695775108701907778035499985460101875335446575668682003476769905757221613059456715124340370004644182419983461938585465709868051843528334865637783432243121141960741183660207253017507840503040137711998706335064580757666737046852760910083551569252122552689413223727032172033758922640655766919875357841645036770531500880515841784135899217450663607891212692876567522590801769697274605349363327350083296511696186865110527484607692711747839048857297143546107068694772796016546186165294713042062143999176884766151223769562343734770798283492672570578271323473514160254230497528088069862319380170640527559784850142194717903077380779552123024550089212802940620650729369703095366866340083842506717784585147853774660975456505047881670007494440012986906121159657431697110316167904065756421094750931686492437056082335336691090542820831077232279552000000000000},\bigskip \\
and the number of self-inverse double cosets is equal to\bigskip \\
\numprint{24134845725866630954648590857135382227340896103570835783490495725848792835299736233031918275420860735380976884874284977933704599570161360906417571942052364586749058508591849043730961252583719618627263189940395063008941322925129995638676303960471120521516312934516546838948169324513923613655921426899202601284705651755393052120843874941795569216757127903621738181229684543309050648396242422250464194464857968143458918823148134173981777643472231778460265343706513960807179556819204471401473349697160766793815864733360814253588040339774426811400832164271876530449280938920797207500979891132990666348857931349519684765495133663024044616026202016095261253578346001171724932645883697778587541917950935343200736065536}.
\end{document}